\documentclass{article}

\usepackage{amsmath,amsthm,amssymb}
\usepackage{comment}

\usepackage[left=25.4mm,right=25.4mm,top=30mm,bottom=25.4mm,a4paper]{geometry}

\theoremstyle{plain}
\newtheorem{theorem}{Theorem}
\newtheorem{corollary}{Corollary}
\newtheorem{lemma}{Lemma}

\theoremstyle{definition}

\newtheorem{remark}{Remark}

\newcommand{\x}{\mathbf{x}}
\newcommand{\y}{\mathbf{y}}
\newcommand{\z}{\mathbf{z}}
\newcommand{\R}{\mathbf{R}}

\title{Convergence and Divergence of Approximations\\in terms of the Derivatives of Heat Kernel}
\author{Jaywan Chung\footnote{Department of Mathematics, Dankook University, 119 Dandae-ro, Dongnam-gu, Cheonan-si, Chungnam 330-714, Korea}}
\date{}

\begin{document}
\maketitle
\begin{abstract}
We consider an approximate solution to the heat equation which consists of the derivatives of heat kernel. Some conditions in the initial value, under which the approximation converges to the solution of the heat equation or diverges when the number of terms of the approximation goes to infinity with a fixed time $t$, will be given. For example, when the initial data is a Gaussian $e^{-|\x|^2/4t_0}$, the approximation converges when $t>t_0$. But if $t<t_0$, it diverges to infinity. Also the $L^\infty$-error estimate will be given and the meaning of the approximation will be clarified by comparing with eigenfunction expansion.
\end{abstract}

\section{Introduction}
Consider the Cauchy problem for the heat equation in the whole space $\R^d$:
\[ \begin{cases}
u_t=\Delta u & \text{in $\R^d \times (0,\infty)$,} \\
u(\x,0)=u_0(\x) & \text{in $\R^d$}.
\end{cases} \]
The problem has a solution
\begin{equation*}
G(\x,t) := (4\pi t)^{-\frac{d}{2}} \exp\Big(-\frac{|\x|^2}{4t}\Big),
\end{equation*}
of which the initial value is a Dirac mass. This solution is called the fundamental solution or the \emph{heat kernel}. Using a linear combination of the spatial derivatives of the heat kernel, we can approximate the solution $u$, as introduced by Duoandikoetxea and Zuazua \cite{MR1183805}. More precisely, it can be shown that \cite[Theorem 4]{MR1183805}:
\begin{theorem}[Duoandikoetxea-Zuazua]
If $1 \leq p < \frac{d}{d-1}$ and $p \leq q \leq \infty$, then there exists a constant $C = C(k,p,q,d) > 0$ such that
\begin{equation*}
\begin{aligned}
\Big\| u(\cdot,t) - \sum_{|\alpha|\leq k} \frac{(-1)^{|\alpha|}}{\alpha!} &\Big(\int \x^\alpha u_0(\x)\,d\x \Big) D^\alpha G(\cdot,t) \Big\|_q \\
& \leq C t^{-\frac{k+1}{2} - \frac{d}{2} \left( \frac{1}{p} - \frac{1}{q} \right)} \big\| |\x|^{k+1} u_0(\x) \big\|_p
\end{aligned}
\end{equation*}
for all $u_0 \in L^1(\R^d; 1+|\x|^k)$ satisfying $|\x|^{k+1}u_0(\x) \in L^p(\R^d)$. Here $k \geq 0$ is an integer, $\alpha=(\alpha_1,\alpha_2,\cdots,\alpha_d)\in\R^d$ is an multi-index, $D^\alpha=\partial_{x_1}^{\alpha_1}\partial_{x_2}^{\alpha_2}\cdots\partial_{x_d}^{\alpha_d}$. Also $L^1(\R^d;1+|\x|^k)$ means the set of all integrable functions with a weight function $1+|\x|^k$.
\end{theorem}
The theorem ensures that the finite sum in terms of the derivatives of the heat kernel converges to the solution when $t\to\infty$. But unfortunately, the constant $C$ in the theorem depends on $k$, which is related to the number of terms to be summed. Hence we cannot determine what happens when the number of terms to be summed goes to infinity, i.e., when $k\to\infty$. In this paper, we will find some conditions under which the series converges to the solution or just diverges when $k\to\infty$ with a fixed time $t$. For example, when the initial data is a Gaussian $e^{-|\x|^2/4t_0}$, the series converges when $t>t_0$ (see Corollary \ref{heat-approx:cor}). But if $t<t_0$, the series diverges to infinity (see Theorem \ref{thm:divergence-approx}). Also in the last section, we will clarify the meaning of the approximation and the conditions by comparing with eigenfunction expansion.

Before proceeding, we introduce some related works. The spatial derivatives of the heat kernel has been also known as \emph{associated functions}. The term emphasizes they are associated with heat polynomials through the Appell transform\footnote{For detailed definitions of the heat polynomials and the associated functions, refer to \cite{Rosenbloom59,Widder61}.}. Rosenbloom and Widder \cite{Rosenbloom59} studied necessary and sufficient conditions under which series expansion of the solution in terms of the heat polynomials and the associated functions is valid \emph{assuming} the series converges. More precisely when $d=1$, they showed that \cite[Theorem 12.3]{Rosenbloom59} a necessary and sufficient condition that
\[u(x,t)=\sum_{k=0}^\infty \frac{(-1)^{k}}{k!} \Big(\int x^k u_0(x)\,dx \Big)  \,\partial_x^k G(x,t),\]
the series converging for $t>\sigma\geq 0$, is that $u(x,t)$ satisfies Huygens principle\footnote{Definition of the Huygens principle, which reminds us the semi-group property, can be found in \cite[Definition 8.1]{Rosenbloom59}. For example, the principle holds for solutions of the Cauchy problem with $L^1$-initial data. } there and that
\begin{equation}\label{cond:Rosen-Widder}
\int_{-\infty}^\infty |u(x,t)| \,e^{x^2/8t}\,dx < \infty \quad \text{for $\sigma < t < \infty$.}
\end{equation}
And since then similar kinds of condition have been obtained in several contexts \cite{Bragg65,Fitouhi89,Haimo65,Widder61}. The integrability condition \eqref{cond:Rosen-Widder} is stated for the solution (not the initial data) and the result focuses on just equality (not convergence of the series). In this paper, we will find convergence and divergence criteria in terms of the initial data and study convergence of the series with an $L^\infty$-error estimate. This approach can be more helpful when we consider the series as an approximate solution to the Cauchy problem because now we can measure the error.

As an application, the approximation considered here can be used to find asymptotic behavior of solutions. For example, we can study long-time asymptotic behavior of the level sets of solutions \cite{Chung12}; the level sets give clear pictures on evolution of solutions.

One may use the heat kernels \emph{having different center of mass}, instead of the derivatives of them. It turns out that this approach is related with a moment problem. Researches in this direction can be found in \cite{MR2600962,MR2773137,KN09}.

\section{Convergence and Divergence Conditions}
In this section, we will find some conditions under which an approximate solution in terms of the derivatives of heat kernel converges or diverges.

The following lemma considers a decomposition of some integrable functions, which was a main point in Duoandikoetxea and Zuazua \cite{MR1183805}. The statement is almost the same as theirs but the difference is that we find a more precise estimate on the remainder \eqref{heat:L1-bound-of-F_alpha}. Although the proof is not very different with one in \cite{MR1183805}, we give a proof for completeness.
\begin{lemma}
  Assume $f \in L^1 (\R^d; 1+|\x|^{k+1} )$ for some integer $k \geq 0$. Then there exists a family of functions $\{ F_\alpha \}$ such that
  \begin{equation} \label{heat:dirac-decomp}
    f = \sum_{|\alpha| \leq k} \frac{(-1)^{|\alpha|}}{\alpha!} \left( \int \x^{\alpha} f(\x) \,d\x \right) D^\alpha \delta + \sum_{|\alpha|=k+1} D^\alpha F_\alpha
  \end{equation}
  in the sense of tempered distribution. Moreover, $F_\alpha \in L^1(\R^d)$ and
  \begin{equation} \label{heat:L1-bound-of-F_alpha}
    \|F_\alpha\|_1 \leq \frac{\| \x^\alpha f(\x) \|_1}{\alpha!}.
  \end{equation}
\end{lemma}
\begin{proof}
  Let $\phi$ be a Schwartz function. By the Taylor's formula,
  \begin{equation*}
    \phi(\x) - (P_k \phi)(\x) = (k+1) \int_0^1 (1-t)^k \sum_{|\alpha|=k+1} \frac{\x^\alpha}{\alpha!} D^\alpha \phi(t\x)\,dt
  \end{equation*}
  where $P_k$ denotes the Taylor polynomial at $\x=0$. Hence
  \begin{align*}
    \Big\langle f - &\sum_{|\alpha| \leq k} \frac{(-1)^{|\alpha|}}{\alpha!} \Big( \int \x^{\alpha} f(\x)\,d\x \Big) D^\alpha \delta, ~\phi \Big\rangle \\
    &=\int f(\x)\phi(\x)\,d\x - \sum_{|\alpha|\leq k} \frac{ (D^\alpha \phi)(0)}{\alpha!} \int \x^\alpha f(\x)\,d\x \\
    &=\int f(\x)\phi(\x)\,d\x - \int f(\x) \sum_{|\alpha|\leq k} \frac{ (D^\alpha \phi)(0)}{\alpha!} \x^\alpha \,d\x \\
    &=\int f(\x) \big( \phi(\x) - (P_k \phi)(\x) \big)\,d\x \\
    &=(k+1) \sum_{|\alpha|=k+1} \int_{\R^d} \int_0^1 f(\x) (1-t)^k \frac{\x^\alpha}{\alpha!} D^\alpha \phi(t\x)\,dt\,d\x \\
    &=\sum_{|\alpha|=k+1} \int_{\R^d} D^\alpha \phi(\x) \int_0^1 (k+1) \frac{(1-t)^k}{t^{d+ |\alpha|}} \frac{\x^\alpha}{\alpha!} f\Big(\frac{\x}{t}\Big) \,dt\,d\x \\
    &=\sum_{|\alpha|=k+1} \langle D^\alpha F_\alpha, ~\phi \rangle,
  \end{align*}
  where
  \begin{equation*}
    F_\alpha(\x) := (-1)^{|\alpha|} \int_0^1 (k+1)\frac{(1-t)^k}{t^{d+|\alpha|}} \frac{\x^\alpha}{\alpha!} f\Big(\frac{\x}{t}\Big)\,dt, \quad k = |\alpha|-1.
  \end{equation*}
  By the Minkowski's inequality for integrals,
  \begin{align*}
    \|F_\alpha\|_1 &\leq \int_0^1 \int_{\R^d} \left| (k+1)\frac{(1-t)^k}{t^{d+|\alpha|}} \frac{\x^\alpha}{\alpha!} f(\x/t) \right|\,d\x\,dt \\
    &= \frac{\| \x^\alpha f(\x) \|_1}{\alpha!} \int_0^1 (k+1) (1-t)^k\,dt = \frac{\| \x^\alpha f(\x) \|_1}{\alpha!},
  \end{align*}
which completes the proof.
\end{proof}

Combininig this lemma, an explicit formula for solutions and an estimate for  Hermite polynomials, we can obtain the following uniform estimate.

\begin{theorem}[Hermite polynomial approximation]
Let $u$ be the solution of the heat equation with an initial data $u_0 \in L^1(\R^d; 1+|\x|^{k+1})$ for some nonnegative integer $k$. Then for all $x \in \R^d$,
\begin{equation} \label{heat:estimate-Hermite-rep}
\begin{split}
\Big| u(\x,t) -& \pi^{-\frac{d}{2}} e^{-\frac{|\x|^2}{4t}} \sum_{|\alpha| \leq k} \frac{\int \x^\alpha u_0(\x)\,d\x}{\alpha!} (4t)^{-\frac{|\alpha|+d}{2}} \prod_{i=1}^d H_{\alpha_i} \Big(\frac{x_i}{2\sqrt{t}}\Big) \Big| \\
&\quad \leq (2\pi)^{-\frac{d}{2}} (2t)^{-\frac{k+d+1}{2}} \sum_{|\alpha|=k+1} \frac{\| \x^\alpha u_0(\x) \|_1}{\sqrt{\alpha!}} \Big( \prod_{i=1}^d (\alpha_i+1) \Big)^{-\frac{1}{12}},
\end{split}
\end{equation}
where the functions $H_n(x)$ are the (physicists') Hermite polynomials defined by
\begin{equation*}
H_n(x) := (-1)^n e^{x^2} \frac{d^n}{dx^n} e^{-x^2}.
\end{equation*}
\end{theorem}
\begin{proof}
The solution $u$ is given by the convolution with the heat kernel
\begin{equation*}
u(\x,t) = [G(\cdot,t) * u_0](\x) = \int G(\x-\y,t) u_0(\y)\,d\y.
\end{equation*}
We put the representation \eqref{heat:dirac-decomp} of $u_0$ here. Then
  \begin{align*}
    u(\x,t) &= \langle u_0, ~G(\x-\cdot,t) \rangle \\
    &= \Big\langle \sum_{|\alpha| \leq k} \frac{(-1)^{|\alpha|}}{\alpha!} \Big( \int \x^{\alpha} u_0(\x)\,d\x \Big) D^\alpha \delta + \sum_{|\alpha|=k+1} D^\alpha F_\alpha, ~G(\x-\cdot,t) \Big\rangle \\
    &= \sum_{|\alpha| \leq k} \frac{(-1)^{|\alpha|}}{\alpha!} \Big( \int \x^\alpha u_0(\x)\,d\x \Big) D^\alpha G(\x,t) + \sum_{|\alpha|=k+1} \big\langle F_\alpha, \,D^\alpha G(\x-\cdot,t) \big\rangle.
  \end{align*}
  Hence by Young's inequality,
\[ \begin{split}
      &\Big\lVert u(\cdot,t) - \sum_{|\alpha| \leq k} \frac{(-1)^{|\alpha|}}{\alpha!} \Big( \int \x^\alpha u_0(\x)\,d\x \Big) D^\alpha G(\cdot,t) \Big\rVert_\infty \\
      &= \Big\lVert \sum_{|\alpha|=k+1} D^\alpha G(\cdot,t) * F_\alpha \Big\rVert_\infty \leq \sum_{|\alpha|=k+1} \|D^\alpha G(\cdot,t)\|_\infty \, \|F_\alpha\|_1.
\end{split} \]
Now we estimate $\|D^\alpha G(\cdot,t)\|_\infty$. Let $\y=(y_i)=\x/(2\sqrt{t})$. Then the derivatives of the heat kernel can be written as a product of Hermite polynomials:
  \begin{align*}
    D^\alpha_{\x} G(\x,t) &= (4\pi t)^{-\frac{d}{2}} D^\alpha_{\x} \exp\Big(-\frac{|\x|^2}{4t}\Big) \\
    &= (4\pi t)^{-\frac{d}{2}} D^\alpha_{\y} \exp(-|\y|^2) \frac{1}{(2\sqrt{t})^{|\alpha|}} \\
    &= \pi^{-\frac{d}{2}} (4t)^{-\frac{|\alpha|+d}{2}} (-1)^{|\alpha|} \prod_{i=1}^d \big[ H_{\alpha_i}(y_i) e^{-y_i^2} \big]
  \end{align*}
  and we use an estimate for Hermite polynomials (see Bonan and Clark \cite{bonan90})
  \begin{equation*}
    \max_{x \in \R} \,|H_n(x)| e^{-x^2} \leq 2^{\frac{n}{2}} \sqrt{n!} \,(n+1)^{-\frac{1}{12}} \quad \text{for all $n \geq 0$}
  \end{equation*}
  to obtain
  \begin{equation*}
    \|D^\alpha G(\cdot,t)\|_\infty \leq \pi^{-\frac{d}{2}} \, 2^{-\frac{|\alpha|}{2}-d} \, t^{-\frac{|\alpha|+d}{2}} \sqrt{\alpha!} \Big( \prod_{i=1}^d (\alpha_i+1) \Big)^{-\frac{1}{12}}.
  \end{equation*}
Finally we use the $L^1$-estimate of $F_\alpha$ in \eqref{heat:L1-bound-of-F_alpha} to complete the proof.
\end{proof}

We denote by $u_k$ the approximation in terms of the derivatives of heat kernel (or Hermite polynomials):
\begin{equation}\label{eq:approx-form}
\begin{split}
 u_k(\x,t) &:= \sum_{|\alpha|\leq k} \frac{(-1)^{|\alpha|}}{\alpha!} \Big(\int \x^\alpha u_0(\x)\,d\x \Big) D^\alpha G(\x,t)\\
&=\pi^{-\frac{d}{2}} e^{-\frac{|\x|^2}{4t}} \sum_{|\alpha| \leq k} \frac{\int \x^\alpha u_0(\x)\,d\x}{\alpha!} \, (4t)^{-\frac{|\alpha|+d}{2}} \prod_{i=1}^d H_{\alpha_i} \Big(\frac{x_i}{2\sqrt{t}}\Big).
\end{split}
\end{equation}

Now we are ready to show that when the fixed time $t$ is sufficiently large, an $L^1$-initial data bounded by a Gaussian guarantees the $L^\infty$-convergence of the approximation as more terms are summed.

\begin{corollary}[Convergence Condition] \label{heat-approx:cor}
Assume the initial data is bounded by a Gaussian: $|u_0(\x)| \leq C e^{-\frac{|\x|^2}{4t_0}}$ a.e. $\x$ for some positive constants $C$ and $t_0$. Then it holds that
  \begin{equation*}
    \big\| u(\cdot,t) - u_k(\cdot,t) \big\|_\infty  \leq C \Big(\frac{t_0}{t}\Big)^{\frac{k+d+1}{2}}  \Big(1 + \frac{k+1}{d}\Big)^{-\frac{d}{12}} \frac{(k+d)!}{(k+1)!\,(d-1)!} =: G(k).
  \end{equation*}
\end{corollary}

The right-hand side, $G(k)$, in the corollary contains three terms; $(t_0/t)^{\frac{k+d+1}{2}}$ is an exponential function of $k$, $(1+(k+1)/d)^{-\frac{d}{12}}$ is a power function of $k$ and $\frac{(k+d)!}{(k+1)!\,(d-1)!}$ is a polynomial of $k$. Hence if the fixed time $t$ is strictly bigger than $t_0$, the right hand side $G(k)$ goes to zero as $k$ goes to infinity and the approximation converges to the solution. If $d=1$ and $t = t_0$, then $G(k) = C(k+2)^{-\frac{1}{12}}$ so that the approximation still converges but the convergence is very slow.

\begin{proof}[Proof of Corollary \ref{heat-approx:cor}]
  Let $F(k)$ be the right hand side of the estimate \eqref{heat:estimate-Hermite-rep}:
  \begin{equation*}
    F(k) := (2\pi)^{-\frac{d}{2}} (2t)^{-\frac{k+d+1}{2}} \sum_{|\alpha|=k+1} \frac{\| \x^\alpha u_0(\x) \|_1}{\sqrt{\alpha!}} \Big( \prod_{i=1}^d (\alpha_i+1) \Big)^{-\frac{1}{12}}
  \end{equation*}
  We will show that $F(k)$ is bounded by $G(k)$. First we calculate a bound for moments:
  \begin{equation*}
    \| \x^\alpha u_0(\x) \|_1 \leq C \big\| \x^\alpha e^{-\frac{|\x|^2}{4t_0}} \big\|_1 = C (4t_0)^{\frac{|\alpha|+d}{2}} \prod_{i=1}^d \Gamma\Big(\frac{\alpha_i+1}{2}\Big).
  \end{equation*}
  From the duplication formula for Gamma functions \cite[p256]{ab72}
  \begin{equation*}
    \Gamma(z) \,\Gamma\Big(z+\frac{1}{2}\Big) = 2^{1-2z} \sqrt{\pi}\,\Gamma(2z),
  \end{equation*}
  it holds that
  \begin{equation*}
    \Gamma\Big(\frac{\alpha_i+1}{2}\Big)^2 \leq \Gamma\Big(\frac{\alpha_i+1}{2}\Big) \sqrt{\pi} \,\Gamma\Big(\frac{\alpha_i}{2}+1\Big) = 2^{-\alpha_i} \pi \,\Gamma(\alpha_i+1).
  \end{equation*}
  By these inequalities and the inequality of arithmetic and geometric means, we can find the bound of $F(k)$:
  \begin{align*}
    F(k) &\leq C (2\pi)^{-\frac{d}{2}} \Big(\frac{2t_0}{t}\Big)^{\frac{k+d+1}{2}} \sum_{|\alpha|=k+1} \Big(\prod_{i=1}^d \frac{\Gamma\big((\alpha_i+1)/2\big)}{\sqrt{\Gamma(\alpha_i+1)}}\Big)  \Big( \prod_{i=1}^d (\alpha_i+1) \Big)^{-\frac{1}{12}} \\
    &\leq C \left(\frac{t_0}{t}\right)^{\frac{k+d+1}{2}}  \Big(1 + \frac{k+1}{d}\Big)^{-\frac{d}{12}} \sum_{|\alpha|=k+1} 1 \\
    &= C \left(\frac{t_0}{t}\right)^{\frac{k+d+1}{2}}  \Big(1 + \frac{k+1}{d}\Big)^{-\frac{d}{12}} \frac{(k+d)!}{(k+1)!\,(d-1)!} = G(k).
  \end{align*}
\end{proof}

On the other hand, the approximation can diverge for some initial data.

\begin{theorem}[Divergence of the Approximation]\label{thm:divergence-approx}
Assume the initial data is a Gaussian: $u_0(\x) = C e^{-\frac{|\x|^2}{4t_0}}$ a.e. $\x$ for some positive constants $C$ and $t_0$. Then for any fixed time $t$ such that $0 < t < t_0$ we have
  \[ \| u_k(\cdot,t) \|_\infty \rightarrow \infty \quad \text{as $k\to\infty$}. \]
  More precisely, when $d=1$ there is a constant $B>0$ which does not depend on $k$ such that
  \[ |u_k(0,t)| \geq B \Big(\frac{t_0}{t}\Big)^{ \lfloor\frac{k}{2}\rfloor} \sqrt{\frac{1}{\lfloor k/2 \rfloor -1}} \quad \text{for every sufficiently large $k$}. \]
  Also when $d \geq 2$ it holds that
  \[ |u_k(0,t)| \geq \frac{C}{(4t)^{d/2} \Gamma(d/2)} \Big(\frac{t_0}{t} - 1\Big) \Big(\frac{t_0}{t}\Big)^{ \lfloor\frac{k}{2} \rfloor -1} \quad \text{for every $k$}. \]
\end{theorem}
\begin{proof}
Because the initial data is radially symmetric ($u_0(\x) = u_0(r), ~r=|\x|$), the approximation $u_k(\x)$ can be simplified. If $\alpha_i$ is an odd number,
  \begin{align*}
    \int_{-\infty}^\infty x_i^{\alpha_i} u_0(r) \,dx_i &= \int_0^\infty x_i^{\alpha_i} u_0(r) \,dx_i + \int_{-\infty}^0 x_i^{\alpha_i} u_0(r) \,dx_i \\
    & = \int_0^\infty x_i^{\alpha_i} u_0(r) \,dx_i - \int_0^{\infty} x_i^{\alpha_i} u_0(r) \,dx_i = 0.
  \end{align*}
Hence we may assume \emph{every $\alpha_i$'s are even}.
\par First we consider the one-dimensional case $d=1$. Then we have
  \begin{align*}
    \sqrt{\pi} \,u_k(0,t) &= \sum_{\substack{j \leq k \\ j:\text{even}}} \frac{\int x^j u_0(x)\,dx}{j!} (4t)^{-(j+1)/2} H_j(0) \\
    &= C \sum_{\substack{j \leq k \\ j:\text{even}}} \frac{1}{j!} \Big(\frac{t_0}{t}\Big)^{\frac{j+1}{2}} \Gamma\Big(\frac{j+1}{2}\Big) H_j(0) \\
    &= C\sqrt{\frac{\pi t_0}{t}} + C \sum_{\substack{0 < j \leq k \\ j:\text{even}}} \frac{1}{j!} \Big(\frac{t_0}{t}\Big)^{\frac{j+1}{2}} \Gamma\Big(\frac{j+1}{2}\Big) H_j(0) \\
    &= C \sqrt{\frac{\pi t_0}{t}} + C \pi^{-1/2} \sum_{\substack{0 < j \leq k \\ j:\text{even}}} \frac{(-1)^{j/2} }{j!} \Big(\frac{t_0}{t}\Big)^{\frac{j+1}{2}} \Gamma\Big(\frac{j+1}{2}\Big)^2 \,2^j.
  \end{align*}
  We may assume $k$ is even. Also assume $k/2$ is even for now. Then it holds that
  \begin{align*}
    \frac{\sqrt{\pi} (-1)^{k/2}}{C} u_k(0,t) &= (-1)^{k/2} \sqrt{\frac{\pi t_0}{t}} + \pi^{-1/2} \sum_{\substack{0 < j \leq k \\ j:\text{even}}} (-1)^{(j+k)/2} \Big(\frac{t_0}{t}\Big)^{\frac{j+1}{2}} \frac{\Gamma((j+1)/2)^2}{\Gamma(j+1)} \,2^j \\
    &= (-1)^{k/2} \sqrt{\frac{\pi t_0}{t}} + \sum_{\substack{0 < j \leq k \\ j:\text{even}}} (-1)^{(j+k)/2} \Big(\frac{t_0}{t}\Big)^{\frac{j+1}{2}} \frac{\Gamma((j+1)/2)}{\Gamma(j/2 + 1)} \\
    & \text{(by duplication formula)} \\
    &= \sqrt{\frac{\pi t_0}{t}} + \Big(\frac{t_0}{t}\Big)^{\frac{k+1}{2}} \frac{\Gamma((k+1)/2)}{\Gamma(k/2 + 1)} - \Big(\frac{t_0}{t}\Big)^{\frac{k-1}{2}} \frac{\Gamma((k-1)/2)}{\Gamma(k/2)} \\
    & \quad + \cdots + \Big(\frac{t_0}{t}\Big)^{\frac{5}{2}} \frac{\Gamma(5/2)}{\Gamma(3)} - \Big(\frac{t_0}{t}\Big)^{\frac{3}{2}} \frac{\Gamma(3/2)}{\Gamma(2)}.
  \end{align*}
  By Stirling's formula \cite[p257]{ab72}, there exists a constant $B > 0$ such that
  \[ \Big|\frac{\Gamma((k+1)/2)}{\Gamma(k/2 + 1)} - \sqrt{\frac{2}{k}} \Big| = \frac{B}{k} \quad \text{for every sufficiently large $k$}. \]
  Therefore for every sufficiently large $k$,
  \begin{align*}
    \frac{\sqrt{\pi} (-1)^{k/2}}{C} u_k(0,t) &\geq \Big(\frac{t_0}{t}\Big)^{\frac{k+1}{2}} \Big( \sqrt{\frac{2}{k}} - \frac{B}{k} \Big) - \Big(\frac{t_0}{t}\Big)^{\frac{k-1}{2}} \Big( \sqrt{\frac{2}{k-2}} + \frac{B}{k-2} \Big) \\
    &= \Big(\frac{t_0}{t}\Big)^{\frac{k-1}{2}} \sqrt{\frac{2}{k-2}} \Big\{ \frac{t_0}{t} \Big( \sqrt{\frac{k-2}{k}} - \frac{B\sqrt{k-2}}{\sqrt{2}k} \Big) - \Big( 1 + \frac{B}{\sqrt{2(k-2)}} \Big) \Big\} \\
    &\gtrsim \Big(\frac{t_0}{t}\Big)^{\frac{k-1}{2}} \sqrt{\frac{2}{k-2}}.
  \end{align*}
Similar computation works when $k/2$ is odd. We have just proved that for every sufficiently large $k$,
\[ |u_k(0,t)| \gtrsim \Big(\frac{t_0}{t}\Big)^{k/2} k^{-1/2}. \]
\par Next consider the case when $d \geq 2$. Then using hyperspherical coordinates and trigonometric integrals, we can compute
\begin{equation*}
    \frac{1}{\alpha!} \int \x^\alpha u_0(\x)\,d\x = \frac{C(|\alpha|,d)}{2^{|\alpha|/2} \prod_{i=1}^d (\alpha_i/2)!} \int_0^\infty r^{|\alpha|+d-1} u_0(r)\,dr
\end{equation*}
where $C(|\alpha|,d)$ is a nonzero constant defined by
\[ C(|\alpha|,d) := \begin{cases}
    \frac{(2\pi)^{d/2}}{(|\alpha|+d-2) (|\alpha|+d-4) \cdots 2} = \frac{(2\pi)^{d/2}}{2^{(|\alpha|+d-2)/2} \Gamma((|\alpha|+d)/2)} & \text{if $d$ is even}, \\
    \frac{2\,(2\pi)^{(d-1)/2}}{(|\alpha|+d-2) (|\alpha|+d-4) \cdots 1} = \frac{(2\pi)^{(d-1)/2} 2^{(|\alpha|+d+1)/2} \Gamma((|\alpha|+d+1)/2)}{\Gamma(|\alpha|+d)} & \text{if $d$ is odd}.
\end{cases} \]
Detailed computation can be found in \cite[p.713]{Chung12}. Also using generalized Laguerre polynomials \cite[p.775]{ab72}
\[ L_n^{(a)}(x) := \frac{x^{-a}e^{x}}{n!} \frac{d^n}{dx^n} (e^{-x} x^{n+a})\]
and their properties \cite[p.785,p.779]{ab72}
\begin{align*}
L_n^{(a+b+1)}(x+y) &= \sum_{i=0}^n L_i^{(a)}(x) L_{n-i}^{(b)}(y), \\
H_{2n}(x) &= (-1)^n \,2^{2n} \,n! \,L_n^{(-1/2)}(x^2),
\end{align*}
we can compute the approximation $u_k$:
\begin{align*}
    (4\pi t)^{\frac{d}{2}} e^{\frac{|\x|^2}{4t}} u_k(\x,t) &= \sum_{\substack{j \leq k \\ j:\text{even}}} \sum_{|\alpha| = j} (4t)^{-\frac{|\alpha|}{2}}  \frac{\int \x^\alpha u_0(\x)\,d\x}{\alpha!} \prod_{i=1}^d H_{\alpha_i} \Big(\frac{x_i}{2\sqrt{t}}\Big) \\
    &= \sum_{\substack{j \leq k \\ j:\text{even}}} (4t)^{-\frac{j}{2}} \int_0^\infty r^{j+d-1} u_0(r)\,dr \sum_{|\alpha|=j} \frac{C(j,d)}{2^{j/2} \prod_{i=1}^d (\alpha_i/2)!} \prod_{i=1}^d H_{\alpha_i} \Big(\frac{x_i}{2\sqrt{t}}\Big) \\
    &= \sum_{\substack{j \leq k \\ j:\text{even}}} (4t)^{-\frac{j}{2}} \frac{C}{2} (4t_0)^{\frac{j+d}{2}} \Gamma\Big(\frac{j+d}{2}\Big) (-2)^{\frac{j}{2}} C(j,d) \sum_{|\alpha|=j} \prod_{i=1}^d L_{\alpha_i/2}^{(-1/2)}\Big(\frac{x_i^2}{4t}\Big) \\
    &= \frac{C}{2} \sum_{\substack{j \leq k \\ j:\text{even}}} \Big(-\frac{2t_0}{t}\Big)^{\frac{j}{2}} \Gamma\Big(\frac{j+d}{2}\Big) C(j,d) \,L_{j/2}^{\big(\frac{d-2}{2}\big)} \Big(\frac{r^2}{4t}\Big).
  \end{align*}
Also from \cite[p.777]{ab72},
\[ L_n^{(a)}(0) = {n+a \choose n} = \frac{\Gamma(n+a+1)}{\Gamma(n+1) \Gamma(a+1)} \quad \text{if $a \geq 0$}, \]
it holds that
  \begin{align*}
    \frac{2}{C} (4\pi t)^{\frac{d}{2}} u_k(0,t) &= \sum_{\substack{j \leq k \\ j:\text{even}}} \Big(-\frac{2t_0}{t}\Big)^{\frac{j}{2}} \Gamma\Big(\frac{j+d}{2}\Big) C(j,d) \,L_{j/2}^{\big(\frac{d-2}{2}\big)}(0). \\
    &= \sum_{\substack{j \leq k \\ j:\text{even}}} \Big(-\frac{2t_0}{t}\Big)^{\frac{j}{2}} \frac{\Gamma((j+d)/2)^2}{\Gamma(j/2 + 1) \Gamma(d/2)} C(j,d).
  \end{align*}
We may assume that $k$ is even. When the dimension $d$ is even and $k/2$ is even, we have
\begin{align*}
    (-1)^{k/2} \Gamma(d/2) &\frac{(4t)^{d/2}}{C} u_k(0,t) = \sum_{\substack{j \leq k \\ j:\text{even}}} (-1)^{\frac{j+k}{2}} \Big(\frac{t_0}{t}\Big)^{\frac{j}{2}} \frac{\Gamma((j+d)/2)}{\Gamma(j/2 + 1)} \\
    &= \Big(\frac{t_0}{t}\Big)^{\frac{k}{2}} \frac{\Gamma((k+d)/2)}{\Gamma(k/2 + 1)} - \Big(\frac{t_0}{t}\Big)^{\frac{k}{2}-1} \frac{\Gamma((k+d)/2 - 1)}{\Gamma(k/2)} + \cdots + (-1)^{k/2} \Gamma(d/2) \\
    &\geq \Big(\frac{t_0}{t}\Big)^{\frac{k}{2}} \frac{\Gamma((k+d)/2)}{\Gamma(k/2 + 1)} - \Big(\frac{t_0}{t}\Big)^{\frac{k}{2}-1} \frac{\Gamma((k+d)/2 - 1)}{\Gamma(k/2)} \\
    & \qquad + \cdots + \Big(\frac{t_0}{t}\Big)^2 \frac{\Gamma(d/2 + 2)}{3} - \Big(\frac{t_0}{t}\Big) \frac{\Gamma(d/2 + 1)}{\Gamma(2)} \\
    &\geq \Big(\frac{t_0}{t}\Big)^{\frac{k}{2}} - \Big(\frac{t_0}{t}\Big)^{\frac{k}{2}-1} + \cdots + \Big(\frac{t_0}{t}\Big)^2 - \Big(\frac{t_0}{t}\Big) \\
    &= \Big(\frac{t_0}{t} - 1\Big) \Big\{ \Big(\frac{t_0}{t}\Big)^{\frac{k}{2}-1} + \Big(\frac{t_0}{t}\Big)^{\frac{k}{2}-3} + \cdots + \frac{t_0}{t} \Big\} \\
    &\geq \Big(\frac{t_0}{t} - 1\Big) \Big(\frac{t_0}{t}\Big)^{\frac{k}{2}-1}.
\end{align*}
The same result holds when $d$ is even and $k/2$ is odd. Now assume the dimension $d$ is odd. Then we have
  \begin{align*}
    (-1)^{k/2} \Gamma(d/2) &\frac{2\sqrt{\pi} t^{d/2}}{C} u_k(0,t) = \sum_{\substack{j \leq k \\ j:\text{even}}} (-1)^{\frac{j+k}{2}} 2^j \Big(\frac{t_0}{t}\Big)^{\frac{j}{2}} \frac{\Gamma((j+d)/2)^2 \Gamma((j+d+1)/2)}{\Gamma(j/2+1) \Gamma(j+d)} \\
    &= 2^{1-d} \sqrt{\pi} \sum_{\substack{j \leq k \\ j:\text{even}}} (-1)^{\frac{j+k}{2}} \Big(\frac{t_0}{t}\Big)^{\frac{j}{2}} \frac{\Gamma((j+d)/2)}{\Gamma(j/2+1)} \quad \text{(by duplication formula)} \\
    &\geq 2^{1-d} \sqrt{\pi} \Big(\frac{t_0}{t} - 1\Big) \Big(\frac{t_0}{t}\Big)^{\frac{k}{2}-1} \quad \text{(by the same argument as above)}.
  \end{align*}
Summing up, we have just proved that when $d \geq 2$,
\[ |u_k(0,t)| \geq \frac{C}{(4t)^{d/2} \Gamma(d/2)} \Big(\frac{t_0}{t} - 1\Big) \Big(\frac{t_0}{t}\Big)^{\frac{k}{2}-1}, \]
which completes the proof.
\end{proof}

\begin{remark}
Kim and Ni \cite{KN09} \emph{conjectured} that if the dimension is $d=1$ and the initial data is a Gaussian, $u_0(x) = \frac{1}{\sqrt{4\pi t_0}} e^{-x^2/4t_0}$, it holds that
\begin{equation*}
    \lim_{k \rightarrow \infty} \frac{\| u(\cdot,t) - u_{k+2}(\cdot,t) \|_\infty}{\| u(\cdot,t) - u_{k}(\cdot,t) \|_\infty} = \frac{t_0}{t},
  \end{equation*}
where $u_k(x,t)$ is the approximation \eqref{eq:approx-form} in one spatial dimension. This is an open problem.
\end{remark}

\section{Relation with Eigenfunction Expansion}
In this section, we relates the approximation \eqref{eq:approx-form} with eigenfunction expansion. Then it will be clarified that there is a close relationship between them, and between their convergence criterion.

Following \cite{MR1491842}, we define similarity variables:
\[ u(\x,t) = t^{-\frac{d}{2}} U(\mathbf{z},\tau), \quad \mathbf{z} = \frac{\mathbf{x}}{2 \sqrt{t}}, \quad \tau = \ln t. \]
Then the heat equation is transformed to
\begin{equation} \label{heat:similarity-eq}
U_\tau = \frac{1}{4} \sum_{i=1}^d \frac{\partial}{\partial z_i} \left( e^{-z_i^2} \frac{\partial}{\partial z_i} \left[ e^{z_i^2} U \right] \right).
\end{equation}
We consider an eigenvalue problem corresponding to \eqref{heat:similarity-eq}. Let $\phi_n(z) , \,n \geq 0$ be eigenfunctions satisfying
\[ \frac{1}{4} \frac{d}{dz} \left( e^{-z^2} \frac{d}{dz} \left[ e^{z^2} \phi_n(z) \right] \right) = \lambda_n \phi_n(z). \]
This equation can be rewritten in a self-adjoint form using change of variables $\widetilde{\phi}_n(z) := e^{z^2} \phi_n(z)$. Then the solution is $\widetilde{\phi}_n(z) = H_n(z)$, an Hermite polynomial, with eigenvalue $\lambda_n = -n/2$. Hence the eigenfunction is $\phi_n(z) = e^{-z^2} H_n(z)$ and eigenfunction expansion for the solution $U$ is
\begin{align}
U(\z,\tau) &= \sum_{\alpha \geq 0} a_\alpha \prod_{i=1}^d \phi_{\alpha_i}(z_i) \,e^{\lambda_{\alpha_i} \tau}  \nonumber \\
&= e^{-|\z|^2} \sum_{\alpha \geq 0} \, a_{\alpha} \,e^{-\frac{|\alpha|}{2} \tau} \prod_{i=1}^d H_{\alpha_i}(z_i). \label{heat:eigenfunction-expansion}
\end{align}
To evaluate the coefficients $a_\alpha$, fix $\tau = \tau_0 > -\infty$, multiply both sides of \eqref{heat:eigenfunction-expansion} by $\z^\alpha$ and integrate. Then the orthogonality relation
\[ \int_{-\infty}^\infty H_m(z)\, z^n\, e^{-z^2}\,dz = n!\,\sqrt{\pi}\,\delta_{mn} \]
yields that
\begin{equation*}
\int \z^\alpha \, U(\z,\tau_0)\,d\z = a_\alpha \, e^{-\frac{|\alpha|}{2}\tau_0} \alpha! \, \pi^{\frac{d}{2}}.
\end{equation*}
Hence the coefficients are
\[ \begin{split}
a_\alpha &= \frac{\pi^{-\frac{d}{2}}}{\alpha!} e^{\frac{|\alpha|}{2} \tau_0} \int \z^\alpha \, U(\z, \tau_0)\,d\z \\
&= \frac{2^{-|\alpha|-d} \,\pi^{-\frac{d}{2}}}{\alpha!} \int \x^\alpha u(\x,t_0)\,d\x =: a_\alpha(t_0).
\end{split} \]
Note that in the coefficients $a_\alpha$, the time $t_0 = e^{\tau_0} > 0$ is fixed. With these coefficients $a_\alpha = a_\alpha(t_0)$, the eigenfunction expansion \eqref{heat:eigenfunction-expansion} is a solution for $t \geq t_0 > 0$.
\par On the other hand, the approximation \eqref{eq:approx-form} can be written in the similarity variables as
\[ U(\z,\tau) \approx e^{-|\z|^2} \sum_{\alpha \geq 0} \, a_\alpha(0) \,e^{-\frac{|\alpha|}{2} \tau} \prod_{i=1}^d H_{\alpha_i}(z_i). \]
Hence the approximation is essentially the same as the eigenfunction expansion. The only difference is that \emph{the approximation uses the moments of the initial data instead of the moments at a positive time.}

Finally we examine when the eigenfunction expansion \eqref{heat:eigenfunction-expansion} is valid. Multiplying both sides of \eqref{heat:eigenfunction-expansion} by $e^{|\z|^2}$ we have
\begin{equation} \label{heat:Hermite-expansion}
e^{|\z|^2} U(\z,\tau) = \sum_{\alpha \geq 0} \, a_{\alpha} \,e^{-\frac{|\alpha|}{2} \tau} \prod_{i=1}^d H_{\alpha_i}(z_i).
\end{equation}
Recall that Hermite polynomials form an orthogonal basis for the Hilbert space of functions $f(z)$ satisfying
\[ \int_{-\infty}^\infty |f(z)|^2 \, e^{-z^2}\,dz < \infty. \]
For fixed time $\tau > -\infty$, equation \eqref{heat:Hermite-expansion} is an expansion of $e^{|\z|^2} U(\z,\tau)$ in terms of Hermite polynomials so that the equation is valid when
\[ \int \big( e^{|\z|^2} U(\z,\tau) \big)^2 e^{-|\z|^2}\,d\z = \int e^{|\z|^2} U^2(\z,\tau) \,d\z < \infty. \]
If the inital data is given by $u_0(\x) = C e^{-\frac{|\x|^2}{4t_0}}$ for some positive constants $C$ and $t_0$, then the solution is
\[ u(\x,t) = \frac{C}{\big( 4\pi\,(t + t_0) \big)^{d/2}} e^{-\frac{|\x|^2}{4(t + t_0)}} \]
or in the similarity variables
\[ U(\z,\tau) = \frac{C}{\big( 4\pi\,(1 + t_0/t) \big)^{d/2}} e^{-\frac{|\z|^2}{1 + t_0/t}}. \]
Hence 
\[ \int e^{|\z|^2} U^2(\z,\tau) \,d\z = \frac{C^2}{\big( 4\pi (1+t_0/t) \big)^d} \int \exp\Big\{ \Big( 1 - \frac{2}{1+t_0/t} \Big) |\z|^2 \Big\} \,d\z \]
is finite if and only if $1-\frac{2}{1+t_0/t} < 0$ or
\[t > t_0,\]
which agrees with a convergence condition for our approximation in Corollary \ref{heat-approx:cor}.

\section*{Acknowlegements}
This research was supported by the 2014 research grant of Dankook University.

\end{document}